\documentclass{amsart}
\usepackage{amsfonts,amssymb,amsmath,amsthm}
\usepackage{url}
\usepackage{enumerate}

\usepackage[utf8]{inputenc}
\usepackage[T1]{fontenc}
\usepackage[english]{babel}

\usepackage{amsmath}
\usepackage{amsthm}
\usepackage{amssymb}

\usepackage[bookmarks=true,
            bookmarksnumbered=true,
            pagebackref=true,
            colorlinks=true,
            linkcolor=blue,
            citecolor=red]{hyperref}

\newcommand{\beq}{\begin{eqnarray}}
\newcommand{\eeq}{\end{eqnarray}}

\newcommand\Nat{\mathbb{N}}
\newcommand\RCA{\mathsf{RCA}}
\newcommand\ACA{\mathsf{ACA}}

\newcommand\ATR{\mathsf{ATR}}

\newcommand\RT{\mathsf{RT}}

\newcommand\CA{\mathsf{CA}}

\newcommand\FUT{\mathsf{FUT}}

\newcommand\WO{\mathsf{WO}}
\newcommand\WOP{\mathsf{WOP}}

\newcommand{\card}{\textrm{card}}

\newcommand\W{\mathrm{W}}
\newcommand\sW{\mathrm{sW}}

\newcommand{\LoX}{\mathcal{X}}

\newcommand\PP{\mathsf{P}}
\newcommand\Q{\mathsf{Q}}

\newcommand\boldomega{\pmb{\omega}}
\newcommand\boldepsilon{\pmb{\varepsilon}}

\newcommand\additionalsymbol{\text{\small{\#}}}

\newtheorem{proposition}{Proposition}
\newtheorem{theorem}{Theorem}
\newtheorem{corollary}{Corollary}

\newtheorem{lemma}{Lemma}

\newtheorem{claim}{Claim}
\newtheorem{open.problem}{Open Problem}

\newtheorem*{theorem*}{Theorem}
\newtheorem*{corollary*}{Corollary}
\newtheorem*{proposition*}{Proposition*}
\newtheorem*{lemma*}{Lemma}
\newtheorem*{fact*}{Fact}
\newtheorem*{claim*}{Claim}
\newtheorem*{open.problem*}{Open Problem}
\newtheorem*{remark*}{Remark}
\newtheorem*{example*}{Example}
\newtheorem*{exercise*}{Exercise}

\title{Reductions of Well-Ordering Principles to Combinatorial Theorems}

\author{Lorenzo Carlucci} \address[L.~Carlucci]{Department of Mathematics\\ Sapienza University of Rome\\ Rome, Italy} \email[L.~Carlucci]{lorenzo.carlucci@uniroma1.it} 
\author{Leonardo Mainardi} \address[L.~Mainardi]{Department of Computer Science\\ Sapienza University of Rome\\ Rome, Italy} \email[L.~Mainardi]{leonardo.mainardi@uniroma1.it}
\author{Konrad Zdanowski} \address[K.~Zdanowski]{Institute of Computer Science\\ Cardinal Stefan Wyszynski University\\ Warsaw, Poland} \email[K.~Zdanowski]{k.zdanowski@uksw.edu.pl}

\begin{document}

\maketitle

\begin{abstract}
A well-ordering principle is a principle of the form: If $\LoX$ is well-ordered then $\mathcal{F}(\LoX)$ is well-ordered, where $\mathcal{F}$ is some natural operator transforming linear orders into linear orders. Many important subsystems
of Second-order Arithmetic of interest in Reverse Mathematics are known to be equivalent to well-ordering principles. 

We give a unified treatment for proving lower bounds on the logical strength of various Ramsey-theoretic principles relations using characterizations of the corresponding formal systems in terms of well-ordering principles.

First we combine Girard's characterization of $\ACA_0$ by the well-ordering principle for base-$\omega$ exponentiation with a colouring used by Loebl and Ne\v{s}etril for the analysis of the Paris-Harrington principle to obtain a short combinatorial proof of $\ACA_0$ from Ramsey Theorem for triples. 

We then extend this approach to Ramsey's Theorem for all finite dimensions and $\ACA_0'$, 
using a characterization of this system 
in terms of the well-ordering preservation principle for iterated base-$\omega$ exponentiation 
due to Marcone and Montalb\'{a}n and, independently, to Afshari and Rathjen. 

We then apply this 
method to $\ACA_0^+$ and to an extension of Ramsey's Theorem for coloring relatively 
large sets due to Pudl\`ak and R\"odl and, independently, to Farmaki and Negrepontis, again using 
the characterization of $\ACA_0^+$ in terms of the well-ordering principle for the $\varepsilon$-function
due to Marcone and Montalb\'{a}n and, independently, to Afshari and Rathjen.

Finally we apply the method to Hindman's Finite Sums Theorem for sums of one or two elements
and the well-ordering-principle at the level of $\ACA_0$.

Our implications (over $\RCA_0$) from combinatorial theorems to $\ACA_0$ and $\ACA_0^+$ also establish uniform computable reductions of the corresponding well-ordering principles to the corresponding Ramsey-type theorems.
\end{abstract}

\section{Introduction}
The proof-theoretic and computability-theoretic strength of Ramsey's Theorem have been intensively 
studied. By Ramsey's Theorem for $n$-tuples and $c$ colours we here mean the assertion that every colouring of 
the $n$-tuples of $\Nat$ in $c$ colours admits an infinite monochromatic set. We refer to this 
principle by $\RT^n_c$. In the context of Reverse Mathematics \cite{Sim:99}, 
a full characterization is known for the 
Infinite Ramsey Theorem for $n$-tuples with $n\geq 3$. In particular, it is known that for every 
$n\geq 3$, $\RT^n_2$ is equivalent to $\ACA_0$ (i.e., to $\RCA_0$ augmented by the assertion that the 
Turing jump of any set exists) and that $\forall n \RT^n_2$ is equivalent to 
$\ACA_0'$ (i.e., $\RCA_0$ augmented by the assertion that for all $n$ the $n$-th Turing jump
of any set exists); see \cite{}.  

By a {\em well-ordering principle\/} (a.k.a. {\em well-ordering preservation principle}) we mean an assertion of the following form 
$$\forall \LoX (\WO(\LoX)\to \WO(\mathcal{F}(\LoX))),$$ 
where $\LoX$ is a linear order, $\WO(\LoX)$ is the $\Pi_1^1$ sentence expressing that $\LoX$ is a well-ordering, and $\mathcal{F}$ is an operator from linear orders to 
linear orders. We abbreviate this statement as $\WOP(\LoX\to \mathcal{F}(\LoX))$. 

An old result of
Girard \cite{Gir:87} (see also \cite{Hir:94} for a proof) 
shows that $\ACA_0$ is equivalent to the well-ordering principle $\WOP(\LoX\to \boldomega^\LoX)$, 
where $\boldomega^\LoX$ is the set 
$$\{ (x_0, \dots, x_k) \,:\, x_k \geq_\LoX \dots \geq_\LoX x_1 \geq_\LoX x_0 \}$$ endowed with the lexicographic ordering. 

Well-ordering principles have 
attracted new interest in recent decades. Analogues of Girard's result for systems stronger that $\ACA_0$ 
have been obtained by Afshari and Rathjen \cite{Afs-Rat:09} and by Rathjen and Weiermann \cite{Rat-Wei:11}
using proof-theoretic methods and by Marcone and Montalb\'{a}n 
\cite{Mar-Mon:11} using computability-theoretic methods. 
The systems $\ACA_0'$, $\ACA_0^+$, $\Pi_{\omega^\alpha}^0$-$\CA_0$
and $\ATR_0$ have been proved equivalent to well-ordering principles of increasing strength. 
Current research by Anton Freund and Michael Rathjen is pushing the boundaries to much higher levels of 
logical strength \cite{Rat:pre}.

In this paper we present a method for obtaining implications and reductions from Ramsey-theoretic
statements to well-ordering principles at the level of $\ACA_0$, $\ACA_0'$ and $\ACA_0^+$.
For the reader's convenience we recall the definitions of these systems: 
$$\ACA_0 = \RCA_0 + \forall X \exists Y ( Y = (X)')$$
$$\ACA_0' = \RCA_0 + \forall n\forall X \exists Y ( Y = (X)^{(n)})$$
$$\ACA_0^+ = \RCA_0 + \forall X \exists Y ( Y = (X)^{(\omega)}).$$

First we combine Girard's characterization of $\ACA_0$ and Marcone-Montalb\`an's characterization 
of $\ACA_0'$ in terms of well-ordering principles with some colourings inspired by Loebl and Ne\v{s}etril's  
beautiful combinatorial proof of the independence of Paris-Harrington Theorem from Peano Arithmetic
\cite{Loe-Nes:92} to obtain new proofs of the following results.
$$\RCA_0\vdash \RT^3_2 \to \ACA_0,$$
and
$$\RCA_0\vdash \forall n \forall c \,\RT^n_c \to \ACA_0',$$
respectively by showing that
$$\RCA_0\vdash \RT^3_2 \to \WOP(\LoX \to \boldomega^\LoX),$$
and
$$\RCA_0\vdash \forall n \forall c \,\RT^n_c \to \forall n\WOP(\LoX \to \boldomega^{\langle n, \LoX\rangle}),$$
where the latter principle is the well-ordering principle from $\LoX$ to 
the ordering obtained by applying $n$ times the $\boldomega$ operator on $\LoX$. These results already appeared in \cite{Car:12}.

Secondly, we show how to extend this approach to prove that a Ramsey-type theorem for 
bicolorings of {\em exactly large sets} (i.e., sets $S\subseteq\Nat$ such that $|S|=\min(S)+1$) 
due to Pudl\`ak-R\"odl \cite{Pud-Rod:82} and, independently, to 
Farmaki and Negrepontis (see, e.g., \cite{Far-Neg:08}), which we call the Large Ramsey Theorem and denote by 
$\RT^{!\omega}_2$.
This extension was conjectured, but not proved, in \cite{Car:12}.
The effective and proof-theoretic content of the Large Ramsey Theorem 
have been analyzed by Carlucci and Zdanowski in 
\cite{Car:12}, where it is proved to be equivalent to $\ACA_0^+$. We here 
present a completely different proof of the implication from the Large Ramsey Theorem to $\ACA_0^+$, via the well-ordering principle for the operator $\LoX\to \varepsilon_\LoX$, thus obtaining a new proof of the following result.
$$ \RCA_0 \vdash \RT^{!\omega}_2 \to \ACA_0^+,$$
by showing that 
$$\RCA_0\vdash \RT^{!\omega}_2 \to \WOP(\LoX \to \boldepsilon_\LoX),$$
where $\boldepsilon_\LoX$ denotes 
an operator that, when applied to a well ordering, behaves like the standard operator $\varepsilon$ defined on ordinals. 
The principle $\WOP(\LoX \to \boldepsilon_\LoX)$ is proved equivalent to $\ACA_0^+$ in \cite{Mar-Mon:11}, Theorem 1.7.

Finally we give a proof in a similar spirit that Hindman's Theorem restricted to unions of one or two sets
and two colorings implies and strongly computably reduces the well-ordering principle $\WOP(\LoX \to \boldomega^\LoX)$. This, besides giving a new proof of the best lower bound known for the first non-trivial restriction
of Hindman's Theorem (Proposition 3.1 in \cite{CKLZ:20}), also establishes the first direct connection between restrictions of Hindman's Theorem and principles related to transfinite well-orderings.

Besides giving new proofs of known implications, our proofs 
at the level of $\ACA_0$ and $\ACA_0^+$ give uniform computable reductions of the well-ordering principles to the corresponding Ramsey-theoretic theorems. 

We briefly recall the basic concepts concerning computable reducibility. The theorems studied in this paper are $\Pi^1_2$-principles, i.e., principles that can be written in the following form: 
$$ \forall X (I(X) \to \exists Y S(X,Y))$$
where $I(X)$ and $S(X,Y)$ are arithmetical formulas. For principles $\PP$ of this form we call any $X$ that satisfies
$I$ an {\em instance} of $\PP$ and any $Y$ that satisfies $S(X,Y)$ a {\em solution} to $\PP$ for $X$.
We will use the following notions of uniform reducibility between two $\Pi^1_2$-principles $\PP$ and $\Q$ .

\begin{enumerate}
\item $\Q$ is {\em Weihrauch reducible} to $\PP$ (denoted $\Q \leq_\W \PP$) if there exist Turing functionals $\Phi$ and $\Psi$ such 
that for every instance $X$ of $\Q$ we have that $\Phi(X)$ is an instance of $\PP$, and if $\hat{Y}$ is a solution to $\PP$ for $\Phi(X)$
then $\Psi(X\oplus \hat{Y})$ is a solution to $\Q$ for $X$.
\item $\Q$ is {\em strongly Weihrauch reducible} to $\PP$ (denoted $\Q \leq_{\sW} \PP$) if there exist Turing functionals $\Phi$ and $\Psi$ such 
that for every instance $X$ of $\Q$ we have that $\Phi(X)$ is an instance of $\PP$, and if $\hat{Y}$ is a solution to $\PP$ for $\Phi(X)$
then $\Psi(\hat{Y})$ is a solution to $\Q$ for $X$.
\end{enumerate}

We refer the reader to \cite{DDHMS:16} for background, motivation and basic properties of Weihrauch and other reductions, 
which have become of major interest in Computability Theory and Reverse Mathematics in recent years.
In the present paper we only establish positive reducibility results, indicating when the implications of type $\PP\to \Q$ over
$\RCA_0$ are witnessed by Weihrauch or strongly Weihrauch reductions.

To apply the scheme of reductions to $\WOP$s it is convenient to consider the contrapositive form, e.g., 
for $\WOP(\LoX \to \omega^\LoX)$ we consider the following implication:
$$ \forall \LoX (\neg\WO(\boldomega^\LoX) \to \neg \WO(\LoX)),$$
so that the instances of the problem are sequences witnessing that $\boldomega^\LoX$ is {\em not} well-ordered and 
solutions are sequences witnessing that $\LoX$ is not well-ordered. 

In all our examples, moreover, the solution sequence is not only computable from the instance sequence
but consists only of terms that appear as sub-terms of the instance sequence. 


\section{$\RT^3$, $\ACA_0$, and base-$\omega$ exponentiation}

We give a direct combinatorial argument showing that Ramsey's Theorem for $2$-colorings of triples implies 
$\WOP(\LoX\to \boldomega^{\LoX})$. Interestingly, the needed colourings are adapted from Loebl and Ne\v{s}etril's 
\cite{Loe-Nes:92} combinatorial proof of the unprovability of the Paris-Harrington principle
from Peano Arithmetic. Our proof establishes a strong computable reduction of the well-ordering preservation 
principle to Ramsey's Theorem. The proof optimizes in terms of number of colors the similar proof for $3$-colorings of triples by the
first and third authors in \cite{Car:12}.

For linear orders we use the same notations as in \cite{Mar-Mon:11}. In particular, we define the following 
operator $\boldomega^\LoX$ from linear orders to linear orders. Given a linear order $\LoX$, 
$\boldomega^\LoX$ is the set of finite strings $\langle x_0,x_1,\dots,x_k\rangle$ of 
elements of $\LoX$ where $x_0\geq_\LoX x_1\geq_\LoX\dots\geq_\LoX x_k$. 
When $\LoX$ is an ordinal, the intended meaning
of $\langle x_0,x_1,\dots,x_k \rangle$ is $\omega^{x_0}+\dots+\omega^{x_k}$.

The order $\leq_{\boldomega^{\LoX}}$ on $\boldomega^{\LoX}$ is the lexicographic order. Throughout the rest of the paper, for any ordering $\mathcal{Y} = (Y, \leq_{\mathcal{Y}})$, we just use the symbol $\leq$ in place of $\leq_{\mathcal{Y}}$ when there is no risk of ambiguity.

If $\alpha = \langle x_0,x_1,\dots,x_k \rangle \in \boldomega^\LoX$, we define $lh(\alpha) = k+1$ and $e_i(\alpha) = x_i$ for any $i < lh(\alpha)$; also, for any $\beta = \langle y_0,y_1,\dots,y_l \rangle \in \boldomega^\LoX$ different from $\alpha$, we denote by $\Delta(\alpha,\beta)$ the least index at which $\alpha$ and $\beta$ differ, i.e. the minimum $i$ such that $e_i(\alpha) \neq e_i(\beta)$ if $\beta$ is not an initial segment of $\alpha$ or viceversa, otherwise $\Delta(\alpha,\beta) = min(k,l)+1$. If $\alpha = \beta$, we set $\Delta(\alpha,\beta) = 0$.

\smallskip

We show the following.

\begin{theorem}\label{prop:RT3}
Over $\RCA_0$, $\RT^3_2$ implies $\forall \LoX (\WOP(\LoX \to \boldomega^\LoX))$.
Moreover, 
$$\forall \LoX (\WOP(\LoX \to \boldomega^\LoX)) \leq_\W \RT^3_2.$$
\end{theorem}

\begin{proof}
Assume $\RT^3_2$ and, by way of contradiction, suppose $\neg \WO(\boldomega^\LoX)$. We show $\neg \WO(\LoX)$. 
We define a $\sigma$-computable colouring $C^{(\sigma)}:[\Nat]^3\to 2$ with an explicit sequence parameter
$\sigma$ of intended type $\sigma:\Nat\to field(\boldomega^\LoX)$ as follows: 
$$C^{(\sigma)}(i,j,k)=
\begin{cases}
0 & \mbox{if } \Delta(\sigma_i,\sigma_j) > \Delta(\sigma_j,\sigma_k),\\
1 & \mbox{otherwise.} 
\end{cases}
$$
Let $\alpha:\Nat\to field(\boldomega^\LoX)$ be an infinite descending sequence
in $\boldomega^\LoX$. 
Let $H$ be an infinite $C^{(\alpha)}$-homogeneous set. 
Consider $(\beta_i)_{i\in \Nat}$, where $\beta_i = \alpha_{h_i}$ and $H=\{h_0 < h_1 < \dots\}$. We reason by cases.

Case 1. The colour of $C^{(\alpha)}$ on $[H]^3$ is $0$. Then 
$$ \Delta(\beta_0,\beta_1) > \Delta(\beta_1,\beta_2) > \dots$$
Contradiction, since $\Delta(\beta_i,\beta_{i+1})\in \Nat$.


Case 2. The colour of $C^{(\alpha)}$ on $[H]^3$ is $1$. 

If $\Delta(\beta_i,\beta_j)=\Delta(\beta_j,\beta_k)$ then $e_{\Delta(\beta_i,\beta_j)}(\beta_i)> e_{\Delta(\beta_j,\beta_k)}(\beta_j)$, since 
$\beta_i>\beta_j$. 

If $\Delta(\beta_i,\beta_j) < \Delta(\beta_j,\beta_k)$ then 
$e_{\Delta(\beta_i,\beta_j)}(\beta_i) > e_{\Delta(\beta_j,\beta_k)}(\beta_j)$,
since $\beta_i > \beta_j > \beta_k$. 

Thus, in any case
$$ e_{\Delta(\beta_0,\beta_1)}(\beta_0) > e_{\Delta(\beta_1,\beta_2)}(\beta_1) > \dots.$$

In other words, $\alpha':\Nat\to \LoX$ defined by 
$$ i \mapsto e_{\Delta(\alpha_{h_i},\alpha_{h_{i+1}})}(\alpha_{h_i}),$$
is an infinite descending sequence in $\LoX$. 
\end{proof}

We have the following immediate corollary, from Theorem \ref{prop:RT3} and the fact that $\WOP(\LoX \to \boldomega^\LoX)$ implies $\ACA_0$, see \cite{Gir:87, Hir:94}.

\begin{corollary}
Over $\RCA_0$, $\RT^3_2$ implies $\ACA_0$.
\end{corollary}

\section{$\forall n \RT^n$, $\ACA_0'$, and iterated base-$\omega$ exponentiation}

We generalize the result from the previous section to Ramsey's Theorem with internal universal quantification 
over all dimensions. The proof below is different from the one in \cite{Car:12} and allows to read-off a uniform computable reduction.

Given a linear ordering $\LoX$, we define $\boldomega^{\langle 0, \LoX\rangle} = \LoX$, and $\boldomega^{\langle n+1,\LoX\rangle}=\boldomega^{\boldomega^{\langle n, \LoX\rangle}}$.

In the Weihrauch reduction stated in the next theorem, the instances of the problem 
$\forall h \forall \LoX (\WOP(\LoX \rightarrow \boldomega^{\langle h,\LoX \rangle}))$ are pairs $(h,\alpha)$ 
with $h\geq 2$ and $\alpha$ an infinite decreasing sequence in the linear order $\boldomega^{\langle h, \LoX\rangle}$.

\begin{theorem}\label{prop:RTn}
Over $\RCA_0$, $\forall n \forall c \RT^n_c$ implies $\forall h \forall \LoX (\WOP(\LoX \rightarrow \boldomega^{\langle h,\LoX \rangle}))$. Moreover, 
$$\forall h \forall \LoX (\WOP(\LoX \rightarrow \boldomega^{\langle h,\LoX \rangle})) \leq_\W \forall n \forall c \RT^n_c.$$

\begin{proof}
The case $h=0$ is trivial, while the case $h=1$ holds by Theorem \ref{prop:RT3}. 
So, by way of contradiction, let $\alpha : \Nat \rightarrow field(\boldomega^{\langle h, \LoX \rangle})$ be an infinite descending sequence in $\boldomega^{\langle h, \LoX \rangle}$, for some $h \!\geq\! 2$ and for some well ordering $\LoX$. We show how to construct an infinite descending sequence in $\LoX$, hence contradicting $\WO(\LoX)$.

First, let $\sigma: \Nat \to field(\boldomega^{\langle h, \LoX \rangle}) \cup \{\additionalsymbol\}$ and let us denote by $\sigma_j^{(n),I}$, with $I = \{i_0 < \ldots < i_k\}$, $k \geq 1$, $n \leq k$ and $j \in \Nat$, the result of the process of extracting the ``$n$-th comparing exponent'' of $\sigma_j$ using indexes in $I$, i.e. $\sigma_j^{(0),I} = \sigma_j$ and for $m<n$:
\medskip
\begin{equation*}
\begin{split}
&\sigma_j^{(m+1),I} =
\begin{cases}
e_{\Delta(\sigma_j^{(m),I},\,\sigma_{succ_I(j)}^{(m),I})}\big(\sigma_j^{(m),I}\big)& \text{if }j \in\{i_0,\ldots,i_{k-m-1}\} \text{ and } \\& e_{\Delta(\sigma_j^{(m),I},\sigma_{succ_I(j)}^{(m),I})}\big(\sigma_j^{(m),I}\big) \text{ exists,}\\[16pt]
\additionalsymbol& \text{otherwise,}
\end{cases}
\end{split}
\end{equation*}

\smallskip

where $succ_I(j)$ is the least element in $I$ greater than $j$. 
We then denote by $\sigma^{(n),I}$ the sequence $\langle \sigma_j^{(n),I} \ |\ j \in \Nat \rangle$.

\medskip

Now, let $C_1^{(\sigma)}: [\Nat]^3 \rightarrow 3$ be the same colouring defined in the proof of Theorem \ref{prop:RT3}, with the additional property that $C_1^{(\sigma)}(i,j,k) = \additionalsymbol$ if at least one out of $\sigma_i, \sigma_j, \sigma_k$ is $\additionalsymbol$.
Also, if $I = \{i_0 < \ldots < i_k\}$ with $k \geq 3$, for any $j=0,\ldots,k-3$, let:
$$
v_j^{(\sigma), I} = (C_1^{(\sigma^{(j),I})}(i_0, i_1, i_2), C_1^{(\sigma^{(j),I})}(i_1, i_2, i_3), \ldots, C_1^{(\sigma^{(j),I})}(i_{k-j-3}, i_{k-j-2}, i_{k-j-1})),
$$

and
$$
w_j^{(\sigma), I} = (C_1^{(\sigma^{(j),I})}(i_1, i_2, i_3), C_1^{(\sigma^{(j),I})}(i_2, i_3, i_4), \ldots, C_1^{(\sigma^{(j),I})}(i_{k-j-2}, i_{k-j-1}, i_{k-j})).
$$


\bigskip

Finally, we can define a colouring $C_h^{(\sigma)} : [\Nat]^{h+2} \rightarrow d(h)$, where $d$ is a primitive recursive function to be read off from the proof.
For $I = \{i_0< \ldots< i_{h+1}\}$, let us define:

\begin{equation*}
C_h^{(\sigma)}(I) =
\begin{cases}
\big(v_0^{(\sigma), I},\ w_0^{(\sigma), I}\big)& \text{if } \neg\big(v_0^{(\sigma), I} = w_0^{(\sigma), I} = (1,\ldots,1)\big)\\[8pt]
\big(v_1^{(\sigma), I},\ w_1^{(\sigma), I}\big)& \text{if } \neg\big(v_1^{(\sigma), I} = w_1^{(\sigma), I} = (1,\ldots,1)\big)\\[8pt]
\ldots&\ldots\\[8pt]
\big(v_{h-2}^{(\sigma), I},\ w_{h-2}^{(\sigma), I}\big)& \text{if } \neg\big(v_{h-2}^{(\sigma), I} = w_{h-2}^{(\sigma), I} = 1\big)\\[8pt]
C_1^{(\sigma^{(h-1),I})}(i_0, i_1, i_2)& \text{otherwise}
\end{cases}
\end{equation*}

where each case of $C_h^{(\sigma)}(I)$ is defined assuming that the conditions describing the previous cases do not hold.

By $\RT^{h+2}_{d(h)}$, let $H$ be an infinite $C_h^{(\alpha)}$-homogeneous set. 
We show how to compute an $\LoX$-descending sequence given $\alpha$ and $H$. Let $\{s_0, s_1, \ldots\}$ be an enumeration of $H$ in increasing order.
We first show that the colour of $H$ must be $2$. To exclude the other cases, we argue as follows. 
Let $I = \{s_{i_0} < \ldots < s_{i_{h+1}}\} \in [H]^{h+2}$, $J = \{s_{i_1} < \ldots < s_{i_{h+2}}\} \in [H]^{h+2}$ and let $(v,w)$ be the colour of $H$, where $v = (v_0,\ldots,v_l)$ and $w = (w_0,\ldots, w_l)$. Also, let $l' = h-l$. Informally, $l'$ is the number of times we ``extract'' the exponent of the elements of any $(h+2)$-tuple in $[H]^{h+2}$ in order to get the colour of the tuple, i.e. $(v,w)$.

\smallskip
 
\emph{Case 1.} The colour is $(v,w)$, with $v \neq w$. This is easily seen to be impossible, since $C_h^{(\alpha)}(I)=(v,w)$ and, by definition of $J$, the first component of the colour of $J$ is exactly $w$. But then $v = w$ should hold by homogeneity.

\smallskip
 
\emph{Case 2.} The colour is $(v, w)$ for some $v = w \neq (1, \ldots, 1)$. We have three subcases.

\smallskip

\emph{Case 2.1.} $v \neq (t, \ldots, t)$, with $t \in \{\additionalsymbol, 0, 1\}$. Then, by hypothesis, we have $v_0=w_0, v_1=w_1, \ldots, v_l=w_l$ and, by definition of the vectors $v$ and $w$, we also have $w_0=v_1, w_1=v_2, \ldots, w_{l-1}=v_l$. Thus we have $v_0 = v_1 = \cdots = v_l$. Contradiction.

\smallskip
 
\emph{Case 2.2.} $v = (\additionalsymbol, \ldots, \additionalsymbol)$. Then $v_{j}^{(\alpha), I} = (\additionalsymbol, \ldots, \additionalsymbol)$ for some $j \leq h-2$. However, no components of $v_{j}^{(\alpha), I}$ can be $\additionalsymbol$, since $\alpha$ is a total strictly descending sequence and $v_{j'}^{(\alpha), I}\! = (1,\ldots,1)$ for any $j' < j$. So this case can not occur.

\smallskip
 
\emph{Case 2.3.} $v = (0, \ldots, 0)$. Then the sequence $$
\big\langle \Delta\big(\alpha^{(l'),\{s_n,\ldots,s_{n+l'}\}}_{s_n}, \alpha^{(l'),\{s_{n+1},\ldots,s_{n+l'+1}\}}_{s_{n+1}}\big) \ \big|\ n \in \Nat \big\rangle
$$
is strictly descending in $\Nat$, contradicting $\WO(\omega)$.

\smallskip
 
Since Case 1 and Case 2 cannot occur, for any $I = \{i_0 < \ldots < i_{h+1}\} \subset H$ we have 
$$C_h^{(\alpha)}(I) = C_1^{(\alpha^{(h-1),I})}(i_0, i_1, i_2) \in \{\additionalsymbol, 0,1\}.$$ 
We can discard colours $\additionalsymbol$ and 0 as in Cases 2.2 and 2.3, so the only possible colour for $H$ is 1.
Therefore, we can construct the sequence $\big\langle \alpha^{(h),\{s_n,\ldots,s_{n+h}\}}_{s_n} \ |\ n \in \Nat \big\rangle$, which is an infinite descending sequence in $\LoX$.
\end{proof}
\end{theorem}

The following corollary is immediate from Theorem \ref{prop:RTn} and Theorem 1.5 in \cite{Mar-Mon:11}. 

\begin{corollary}
Over $\RCA_0$, $\forall n\forall k\RT^n_k$ implies $\ACA_0'$.
\end{corollary}

\section{Large Ramsey's Theorem, $\ACA_0^+$ and the $\varepsilon$-ordering}

We apply the proof-technique from the previous sections to a natural extension of Ramsey's Theorem 
for colourings of a particular family of finite sets of unbounded size. 

By slightly redefining the notion of largeness given by Paris and Harrington in \cite{Har-Par:77}, we call a set $X$ {\em exactly large} if $\card(X) = \min(X)+3$, and denote by $[\Nat]^{!\omega}$ the collection of exactly large sets $X \subseteq \Nat$.
Then, we can state the following principle, which is obtained by just applying our definition of exactly large sets to a well-known result by Pudl\`ak-R\"odl \cite{Pud-Rod:82} and Farmaki \cite{Far-Neg:08}.

\begin{theorem} \label{thm:rtomega}
For every infinite subset $M$ of $\Nat$ and 
for every coloring $C$ of the exactly large 
subsets of $\Nat$ in two colours, there exists an infinite 
set $L\subseteq M$ such that every exactly large subset
of $L$ gets the same color by $C$.
\end{theorem}

We refer to the statement of the above Theorem as $\RT^{!\omega}_2$. It is easy to prove that Theorem \ref{thm:rtomega} is implied (over $\RCA_0$) by the original principle due to Pudl\`ak, R\"odl and Farmaki, and that the latter principle reduces (via a strong Weihrauch reduction) our version of $\RT^{!\omega}_2$. Hence, the analysis carried out in \cite{Car-Zda:RT} also applies to Theorem \ref{thm:rtomega}, which is therefore equivalent to $\ACA_0^+$ over $\RCA_0$. 

Here, we apply the technique from the previous sections to 
give a direct proof of $\RCA_0 \vdash \RT^{!\omega}_2 \to \ACA_0^+$, that also witnesses a Weihrauch reduction. 
The proof hinges on the results from the previous sections (Theorems \ref{prop:RT3} and 
\ref{prop:RTn}) and extends the same pattern of arguments to more complex linear orderings. 
Namely, we use $\RT^{!\omega}_2$ to derive $\forall \LoX (\WOP(\LoX\to \boldepsilon_\LoX))$, which in turn has been proved to be equivalent to $\ACA_0^+$ over $\RCA_0$ by Marcone and Montalb\`an (\cite{Mar-Mon:11}, Theorem 1.7).

We adopt the same notation defined in Definition 2.3 in \cite{Mar-Mon:11} and assume all the terms to be written in normal form. For any $\gamma \in \boldepsilon_\LoX$, let $lh(\gamma)$ be the number of terms of $\gamma$ and let $\gamma_n$ be the $n$-th term of $\gamma$ if $n \in [0,lh(\gamma))$, otherwise $\gamma_n = 0$. Then, we denote by $e_n(\gamma)$ the exponent of $\gamma_n$ if such exponent does exist, otherwise we set $e_n(\alpha) = 0$. Also, if $\delta \in \boldepsilon_\LoX$, we indicate by $\Delta(\gamma,\delta)$ the index of the first term at which $\gamma$ and $\delta$ differs -- or 0 if $\gamma = \delta$.
Finally, notice that $\boldomega$ and $\boldepsilon$ operators are compatible to each other in the sense of Lemma 2.6 in \cite{Mar-Mon:11}, so we can refer to results from the previous sections while dealing with $\boldepsilon$.

\begin{theorem}\label{thm:RTomega}
\ Over $\RCA_0$, $\RT^{!\omega}_2$ implies $\forall \LoX (\WOP(\LoX \rightarrow \boldepsilon_\LoX))$. Moreover, 
$$\forall \LoX (\WOP(\LoX \rightarrow \boldepsilon_\LoX)) \leq_\W \RT^{!\omega}_2.$$

\begin{proof}
Suppose $\WO(\LoX)$ but $\neg\WO(\boldepsilon_\LoX)$. Without loss of generality, we assume $0 \in \LoX$. We define a colouring $C_1^{(\sigma)} : [\Nat]^3 \rightarrow 6$ with an explicit sequence parameter $\sigma$ of intended type $\sigma : \Nat \rightarrow field(\boldepsilon_\LoX) \cup \{\additionalsymbol\}$.  Each case of $C_1^{(\sigma)}$ is defined assuming that the conditions describing the previous cases do not hold.

\begin{equation*}
    C_1^{(\sigma)}(i,j,k) = 
    \begin{cases}
    \additionalsymbol   &\text{if }\sigma_i=\,\additionalsymbol \,\lor\ \sigma_j=\,\additionalsymbol \,\lor\ \sigma_k=\,\additionalsymbol\\
    0   &\text{if }\sigma_i < \varepsilon_0\\
    1   &\text{if }\Delta(\sigma_i,\sigma_j) > \Delta(\sigma_j,\sigma_k)\\
    2   &\text{if } b_{\Delta(\sigma_i,\sigma_j)}(\sigma_i) >_\LoX b_{\Delta(\sigma_j,\sigma_k)}(\sigma_j)\\
    3   &\text{if } ht_{\Delta(\sigma_i,\sigma_j)}(\sigma_i) > ht_{\Delta(\sigma_j,\sigma_k)}(\sigma_j)\\
    4   &\text{otherwise}
    \end{cases}
\end{equation*}

where, for $\gamma \in \boldepsilon_\LoX$ and $n \in \Nat$, $b_n(\gamma) = min\{x \in \LoX\ |\ \gamma_n < \varepsilon_{x+1}\}$, while $ht_n(\gamma) = 0$ if $\gamma_n < \varepsilon_0$ or $\gamma_n = \varepsilon_{b_n(\gamma)}$, otherwise $ht_n(\gamma) = 1 + ht_n(\gamma')$, with $\gamma_n = \omega^{\gamma'} \neq \varepsilon_{b_n(\gamma)}$. Informally, when $\gamma_n \geq \varepsilon_0$, $b_n(\gamma)$ is the largest $x \in \LoX$ such that $\varepsilon_x$ appears in $\gamma_n$, while $ht_n(\gamma)$ is the maximum height at which $\varepsilon_{b_n(\gamma)}$ appears in $\gamma_n$.

\medskip

For any $h \geq 2$, we define $C_h^{(\sigma)} : [\Nat]^{h+2} \rightarrow d(h)$ as in Theorem \ref{prop:RTn}. Notice that, in this case, we use $C_1^{(\sigma)}$ as defined above, hence the condition $\neg\big(v_0^{(\sigma), I} = w_0^{(\sigma), I} = (1,\ldots,1)\big)$ in the definition of $C_h^{(\sigma)}$ must be replaced with $\neg\big(v_0^{(\sigma), I} = w_0^{(\sigma), I} = (4,\ldots,4)\big)$.

\medskip

Lastly, we define a colouring $C^{(\sigma)} : [\Nat]^{!\omega} \rightarrow 2$. 

$C^{(\sigma)}$ is defined as follows:
\begin{equation*}
C^{(\sigma)}(i_0, \ldots, i_{i_0+2}) = 
\begin{cases}
0&\text{if }C_{i_0}^{(\sigma)}(i_1, \ldots, i_{i_0+2}) = 4\\
1&\text{otherwise}

\end{cases}    
\end{equation*}

\smallskip
Now we are ready to prove that, by assuming $\neg\WO(\boldepsilon_\LoX)$, we can construct an infinite descending sequence in $\LoX$, thus contradicting $\WO(\LoX)$.

Let $\alpha : \Nat \rightarrow field(\boldepsilon_\LoX)$ be an infinite descending sequence in $\boldepsilon_\LoX$ and, by $\RT^{!\omega}_2$, let $H = \{h_0 < h_1 < \ldots\}$ be an infinite $C^{(\alpha)}$-homogeneous set. Here, by $\RT^{!\omega}_2$ we mean the Ramsey Theorem applied to our definition of $[\Nat]^{!\omega}$: however, this version of $\RT^{!\omega}_2$ is clearly equivalent -- over $\RCA_0$ -- to the original version.



First, notice that the $C^{(\alpha)}$-colour of $H$ is 0, otherwise, for any choice of a positive $h \in H$, we could colour $H \setminus [0, h]$ using $C_h^{(\alpha)}$ and, by $\RT^{h+2}$, we would obtain an infinite $C_h^{(\alpha)}$-homogeneous set whose colour is different from 4, hence contradicting the proof of Theorem \ref{prop:RTn}, or rather its version adapted in order to manage the three additional colours of the base colouring $C_1^{(\alpha)}$. We can apply $\RT^{h+2}$ since it is implied by $\RT^{!\omega}_2$.

Now, we slightly redefine the notation used in Theorem \ref{prop:RTn}. Let us denote by $\alpha_i^{(n),H}$, with $i \in H \setminus \{h_0\}$ and $n \leq prec_H(i) = max\{h \in H\ |\ h < i\}$, the result of the process of extracting the ``$n$-th comparing exponent'' of $\alpha_i$ using indexes in $H$, i.e. for any $m<n$:
\begin{equation*}
    \alpha_i^{(0),H} = \alpha_i
\end{equation*}
\begin{equation*}
    \alpha_i^{(m+1),H} = e_{\Delta(\alpha_i^{(m),H},\alpha_{succ_H(i)}^{(m),H})}\big(\alpha_i^{(m),H}\big)
\end{equation*}

Each term $\alpha_i^{(n),H}$ is well-defined, since for any choice of $j_0 < j_1 < \ldots < j_{prec_H(i)+1}$ we have $$C_{prec_H(i)}^{(\alpha)}(i, h_{j_0}, \ldots, h_{j_{prec_H(i)+1}})\!=\!4$$ that entails $C_n^{(\alpha)}(i, h_{j_0}, \ldots, h_{j_{n+1}})\!=\!4$ for any $0 < n \leq prec_H(i)$, which in turn implies that the comparing exponent in the definition of $\alpha_i^{(m+1),H}$ does exist.

\medskip

Using this notation, we define a sequence $\tau : \Nat \rightarrow field(\LoX)$ as follows:
\begin{equation*}
    \tau_i = b_0\big(\alpha_{n_i}^{(t_i),H}\big)
\end{equation*}

for any $i \geq 0$, where $t_0=h_0$, $n_0=h_1$ and, for any $j \geq 0$,
\begin{equation*}
    t_{j+1} = t_j + ht_0\big(\alpha_{n_j}^{(t_j),H}\big) + 1
\end{equation*}
\begin{equation*}
    n_{j+1} = succ_H \big(min \big\{ h \in H \ \big|\ h \geq t_{j+1} \big\}\big)
\end{equation*}

\medskip

Notice that each term $\tau_i$ is well-defined, since $t_i \leq prec_H(n_i)$, as required by the definition of $\alpha_{n_i}^{(t_i),H}$.

Finally, since the sequence $\big(\alpha_{n_{i+k}}^{(t_i),H}\big)_{k \in \Nat}$ is decreasing by construction (cf. proof of Theorem \ref{prop:RTn}), we have:
\begin{equation*}
\tau_i = b_1\big(\alpha_{n_i}^{(t_i),H}\big) \geq b_1\big(\alpha_{n_{i+1}}^{(t_i),H}\big) \ \overset{(*)}{>_\LoX}\ b_1\big(\alpha_{n_{i+1}}^{(t_{i+1}),H}\big) = \tau_{i+1}
\end{equation*}

where $(*)$ is guaranteed by the choice of $t_{i+1}$: it is indeed large enough to ``lower'' $\alpha_{n_i}^{(t_i),H}$ under $\varepsilon_{\tau_i}$, so it must be large enough to lower $\alpha_{n_{i+1}}^{(t_i),H}$ under $\varepsilon_{\tau_i}$ as well. More precisely, since $\varepsilon_{\tau_i}$ is the maximum $\varepsilon$-term in $\alpha_{n_i}^{(t_i),H}$ and it appears at height $h = ht_0\big(\alpha_{n_i}^{(t_i),H}\big)$, then no terms $\varepsilon_x$ with $x \geq \tau_i$ can appear at height $h'>h$ in $\alpha_{n_{i+1}}^{(t_i),H}$. Hence, no such terms can appear in $\alpha_{n_{i+1}}^{(t_{i+1}),H}$. So $(*)$ holds.

Therefore, $\tau$ is an infinite descending sequence in $\LoX$.

\end{proof}
\end{theorem}

We have the following immediate corollary, yielding an alternative proof of Theorem 3.6 in \cite{Car-Zda:RT}.

\begin{corollary}
Over $\RCA_0$, $\RT^{!\omega}_2$ implies $\ACA_0^+$.
\end{corollary}

\begin{proof}
From Theorem \ref{prop:RTomega} above and Theorem 1.7 in \cite{Mar-Mon:11}, showing that $\WOP(\LoX \to \boldepsilon_\LoX)$ is equivalent to $\ACA_0^+$ over $\RCA_0$.
\end{proof}

\section{Hindman's Theorem and base-$\omega$ exponentiation}

A sequence $\mathcal{B} = (B_i)_{i\in \Nat}$ of finite non-empty subsets of the positive integers is called a {\em block sequence} if
for all $i, j\in \Nat$, if $i < j$ then $\max(B_i) < \min(B_j)$; in this case we write $B_i < B_j$ for short.
We denote by $FU(\mathcal{B})$ the set of finite non-empty unions of elements of $\mathcal{B}$. Hindman's Finite Unions Theorem $\FUT$
(\cite{Hin:74}) states that every finite coloring of the finite non-empty subsets of the positive integers admits an infinite block 
sequence $B$ such that $FU(\mathcal{B})$ is monochromatic. 

The strength of Hindman's Finite Unions Theorem and its restrictions has attracted substantial interest
in recent times (see \cite{Car:overview} for an overview). We establish a connection between Hindman-type
theorems and well-ordering principles, along the lines of our previous results. 

Let $n \geq 1$ and $k \geq 2$. We denote by $FU^{=n}(\mathcal{B})$ the set of unions of $n$ many elements of $\mathcal{B}$, while we use $\FUT^{=n}_k$ to denote the restriction of $\FUT$ based on the number of terms in monochromatic unions for $k$-colourings. Hence, we can state $\FUT^{=n}_k$ as follows: for all colorings of the finite non-empty sets of positive integers
in $k \geq 2$ colours, there exists an infinite block sequence $\mathcal{B} = (B_i)_{i \in \Nat}$ of non-empty finite sets
of positive integers such that all unions of $n$ elements of $\mathcal{B}$ have the 
same colour. 


Since $\FUT^{=3}_2$ imply $\ACA_0$ over $\RCA_0$, (see \cite{CKLZ:20}), we know that 
$\FUT^{=3}_2$ implies $\WOP(\LoX\mapsto \boldomega^\LoX)$ over $\RCA_0$. 

We give a new proof  of this implication (actually, a slightly more general version of it) by a direct argument that furthermore 
establishes a Weihrauch reduction. Its interest also lies in the connection between 
Hindman's Theorem and principles related to transfinite ordinals. 

Let $\LoX$ be a linear ordering. Let $\alpha=(\alpha_i)_{i \in \Nat}$ be an infinite decreasing sequence in 
$\boldomega^\LoX$. We show, using $\FUT^{=n}_k$ for $n \geq 3$ and $k \geq 2$, that there exists an infinite decreasing sequence in 
$\LoX$. 
The proof  uses ideas from the proof of $\FUT^{=3}_2 \to \ACA_0$ (Proposition 3.1, \cite{CKLZ:20}) adapted to the present context, based on the following analogy between deciding 
the Halting Set $K$ and computing an infinite descending sequence in $\LoX$. Given an enumeration of $K$ and a number $n$, $\RCA_0$ knows that there is an $\ell$ such that all numbers in $K$ below $n$ appear within $\ell$ steps of 
the enumeration, but is not able to compute this $\ell$. Similarly, given an ordinal $\alpha$ in an infinite decreasing sequence
in $\boldomega^\LoX$, $\RCA_0$ knows that there is an $\ell$ such that if a term of $\alpha$ ever decreases, it will do so 
by the $\ell$-th term of the infinite descending sequence, but it is unable to compute such an $\ell$. More precisely, 
while one can computably run through the given infinite descending sequence to find the first point at which an exponent 
of a component of $\alpha$ is decreased, we can not locate computably the leftmost such component. An appropriately 
designed coloring will ensure that the information about such an $\ell$ can be read-off the elements of a solution to 
Hindman's Theorem. 

We start with the following simple Lemma.

\begin{lemma}\label{lem:decribile-terms} The following is provable in $\RCA_0$: If $\alpha=(\alpha_i)_{i\in\Nat}$ is an infinite descending sequence in $\boldomega^\LoX$, then
$$\forall n\ \exists n'\ \exists m < lh(\alpha_n) \ \big(n'>n \ \land e_m(\alpha_n) >_{\LoX} e_m(\alpha_{n'})\big).$$
\end{lemma}

\begin{proof}

Assume by way of contradiction that the statement is false, as witnessed by $n$, and recall that for any distinct $\sigma, \tau \in \boldomega^\LoX$, we have $\sigma < \tau$ if and only if either (1.) $\sigma$ is an initial segment of $\tau$, or (2.) there exists $m$ such that $e_m(\sigma) <_\LoX e_m(\tau)$ and  $e_n(\sigma) = e_n(\tau)$ for each $n<m$. Then we can show that:
$$
\forall p\ (p \geq n \rightarrow (\alpha_{p+1} \text{ is an initial segment of both } \alpha_p \text{ and } \alpha_n))
$$
by $\Delta^0_1$-induction.

The case $p=n$ is trivial, since $\alpha_n >_{\LoX} \alpha_{n+1}$ and (2.) cannot hold by assumption.

For $p>n$, by induction hypothesis we know that $\alpha_p$ is an initial segment of $\alpha_n$. 

Since $\alpha_{p+1} >_{\LoX} \alpha_p$, $\ \alpha_{p+1}$ must be an initial segment of $\alpha_p$, otherwise the leftmost component differing between $\alpha_{p+1}$ and $\alpha_p$ -- i.e. the component of $\alpha_{p+1}$ with index $m$ witnessing (2.) -- would contradict our assumption, for we would have $m < lh(\alpha_p)$ and $e_m(\alpha_{p+1}) <_\LoX e_m(\alpha_p) = e_m(\alpha_n)$.

So  $\alpha_{p+1}$ must be an initial segment of $\alpha_p$ and, by our assumption, it must be an initial segment of $\alpha_n$ as well.

The previous statement implies that:
$$
\forall p\ (p \geq n \,\rightarrow\, lh(\alpha_p) > lh(\alpha_{p+1}))
$$

hence contradicting $\WO(\omega)$. This concludes the proof.
\end{proof}

\begin{theorem}\label{prop:RTomega}
Let $n \geq 3, k \geq 2$. Over $\RCA_0$, $\FUT^{=n}_k$ implies $\WOP(\mathcal{X} \to \boldomega^\mathcal{X})$. Moreover, 
$$\WOP(\mathcal{X} \to \boldomega^\LoX) \leq_\W \FUT^{=n}_k.$$
\end{theorem}

\begin{proof}
Assume by way of contradiction $\neg\WO(\boldomega^\LoX)$, and let $\alpha= (\alpha_n)_{n\in \Nat}$ 
be an infinite descending sequence in $\boldomega^\LoX$. 
For this proof, it is convenient to use an $\alpha$-computable sequence $\beta$ of all the components of the terms $\alpha_n$, enumerated in order of ``appearance'', i.e. $\beta = \langle e_0(\alpha_0), e_1(\alpha_0), \dots, e_{lh(\alpha_0)-1}(\alpha_0), e_0(\alpha_1), e_1(\alpha_1), \dots \rangle$. Formally we construct such sequence by first defining $\theta : \Nat\times\Nat\to \Nat$ as follows: $\theta(n,m)= m + \sum_{k < n} lh(\alpha_k)$. The partial function $\theta$ is clearly bijective (meaning that it is a bijection between its domain of definition and its codomain). We accordingly fix functions $t: \Nat \to \Nat$ and $p:\Nat\to\Nat$ such that for each $n\in\Nat$ we have
$\theta(t(n),p(n))=n$. The sequence $(\beta_h)_{h\in\Nat}$ of components of terms in $\alpha$ is then defined
by setting $\beta_h = \alpha_{t(h),p(h)}$. Intuitively, $t(h)$ is the element of $\alpha$ from which $\beta_h$ has been ``extracted'', while $p(h)$ is the position of the component $\beta_h$ within $\alpha_{t(h)}$.


We call $i$ \emph{decreasible} if there exists $j>i$ such that $p(j)=p(i)$ and $\beta_j > \beta_i$. In that case, we say that $j$ \emph{decreases} $i$ and that $j$ is a \emph{decreaser} of $i$. 

Using this terminology, Lemma \ref{lem:decribile-terms} states that each element of $\alpha$ contains at least one decreasible component.

Also, we define a number $j\in [0,r]$ \textit{important in} $S=\{n_0, \ldots,n_r\}$ if the following condition holds:
$$
(\exists i< n_0)\ \,\exists i' \in [n_{j-1}, n_j)\, \text{ s.t. } i' \text{ decreases } i \text{  and } \neg\exists i'' < n_{j-1}\, \text{  s.t. } i'' \text{  decreases } i,$$
where we set  $n_{-1}=0$.
\medskip

Now suppose that $f:\Nat \to \Nat$ is a function with the following property:

\bigskip
{\em Property P}: For all $i\in \Nat$,
if $i$ is decreasible, then it is decreased by some $j \leq f(i)$. 

\bigskip
We first show that given such an $f$ we can compute (in $f$ and $\beta$) 
an infinite descending sequence $(\sigma_i)_{i\in\Nat}$ in $\mathcal{X}$ as follows. 

\medskip
{\bf Step $0$}.  
Let $i_0$ be the least decreasible index of $\beta$, and let $j_0$ be the least decreaser of $i_0$. By Lemma \ref{lem:decribile-terms}, $\beta_{i_0}$ must be a component of $\alpha_0$, i.e. $t(i_0) = 0$, so we can find $i_0$ by just taking the least decreasible $i^* < lh(\alpha_0)$. Notice that we can decide whether $i^*$ is decreasible by just inspecting $\beta$ up to the index $f(i^*)$, since $f$ has the Property $P$.

We set $\sigma_0 = \beta_{j_0}$ and observe that $p(i) \geq p(i_0)$ for each decreasible $i>i_0$. 
Suppose otherwise as witnessed by $i^*$, and let $j^*$ be the least decreaser of $i^*$. By definition of decreaser, $p(j^*) = p(i^*) < p(i_0)$ and $\beta_{i^*} > \beta_{j^*}$. However $i_0$ is the least decreasible $\beta$-index of some component of $\alpha_0$, then $\beta_{i^*} = \beta_z$, where $z = \theta(0,p(i^*))$. Hence, $\beta_z > \beta_{j^*}$ and $p(z)=p(i^*)=p(j^*)$, but in that case $z$ would be decreasible (by $j^*$) and $p(z) < p(i_0)$, which implies $z < i_0$ since $t(z)=t(i_0)=0$, thus contradicting the minimality of $i_0$.

\bigskip
{\bf Step $s+1$}. Suppose $i_s, j_s, \sigma_s$ are defined, $(\sigma_t)_{t \leq s}$ is decreasing in $\mathcal{X}$ and 
$p(i) \geq p(i_s)$ for each decreasible $i > i_s$. 

Let $i_{s+1}$ be the least decreasible index of $\beta$ larger than or equal to $j_s$, and let $j_{s+1}$ be the least decreaser of $i_{s+1}$. By Lemma \ref{lem:decribile-terms} and the fact that no decreasible $i > i_s$ can have $p(i) < p(i_s)=p(j_s)$, $\beta_{i_{s+1}}$ must be a component of $\alpha_{t(j_s)}$, namely the leftmost whose $\beta$-index is decreasible. So we can find $i_{s+1}$ 
by just taking the least decreasible $i^* \in [j_s,\ j_s + lh(\alpha_{t(j_s)}) - p(j_s))$. Notice that we can decide whether $i^*$ is decreasible by just inspecting $\beta$ up to the index $f(i^*)$, since $f$ has the Property $P$.

We set $\sigma_{s+1} = \beta_{j_{s+1}}$. As we noted above, $\beta_{i_{s+1}}$ must be either $\beta_{j_s}$ or a component of $\alpha_{t(j_s)}$ on the right of $\beta_{j_s}$, i.e. $t(i_{s+1}) = t(j_s)$ and $p(i_{s+1}) \geq p(j_s)$, so $\beta_{j_s} \geq \beta_{i_{s+1}}$. Then, $\sigma_s > \sigma_{s+1}$ because $\sigma_s = \beta_{j_s} \geq \beta_{i_{s+1}} > \beta_{j_{s+1}} = \sigma_{s+1}$. Finally, we observe that the last part of the inductive invariant is guaranteed as well, since $p(i) \geq p(i_{s+1})$ for each decreasible $i>i_{s+1}$. 
Suppose otherwise as witnessed by $i^*$, and let $j^*$ be the least decreaser of $i^*$. By definition of decreaser, $p(j^*) = p(i^*) < p(i_{s+1})$ and $\beta_{i^*} > \beta_{j^*}$. However $\beta_{i_{s+1}}$ is the leftmost component of $\alpha_{t(j_s)}$ whose $\beta$-index is decreasible, then $\beta_{i^*} = \beta_z$, where $z = \theta({t(j_s)},p(i^*))$. Hence, $\beta_z > \beta_{j^*}$, but in that case $z$ would be decreasible (by $j^*$) and $p(z) = p(i^*) < p(i_{s+1})$, which implies $z < i_{s+1}$ since $t(z)=t(j_s)=t(i_{s+1})$, thus contradicting the minimality of $i_{s+1}$.

\bigskip
We now show how to obtain a function satisfying the Property $P$ from a solution of $\FUT^{=n}_k$ for
a suitable colouring. Let $g: FIN(\Nat^+) \rightarrow k$ as follows:
$$
g(S) = card\{j\ |\ j \text{ is important in }  S\} \text{ mod } k.
$$

\medskip

By $\FUT^{=n}_k$ let $\mathcal{B} = \{B_0 < B_1 < B_2 < \ldots\}$ be an infinite block sequence such that $FU^{=n}(\mathcal{B})$ is monochromatic under $g$, and let $c<k$ be the colour of $\mathcal{B}$.

\begin{claim} \label{clm:shorter-union}
Given $S_0\!<\!\ldots\!<\!S_{n-3}$ in $\mathcal{B}$, there exists $S \in \mathcal{B}$ such that $S_{n-3} < S$ and $g(S_0 \cup \ldots \cup S_{n-3} \cup S) = c$.
\end{claim}

Fix $S_0<\ldots<S_{n-3}$ in $\mathcal{B}$ and let $\ell$ be the actual upper bound of the minimal indexes decreasing all the decreasible $j < \min(S_0)$. By this we mean the $\ell$ given by the following instance of strong $\Sigma^0_1$-bounding (in $\RCA_0$): 
$$
\forall m \,\exists \ell \,\forall j < m \,( \exists d \,(d \text{ decreases } j) \to \exists d < \ell \,(d \text{ decreases } j))),
$$
where we can take $m=\min(S_0)$. 

Since $\mathcal{B}$ is an infinite block sequence, there exists $S \in \mathcal{B}$ such that $\min(S) > \ell > \max(S_{n-3})$. Then, for any $T \in \mathcal{B}$ with $S < T$, we have that $g(S_0 \cup \ldots \cup S_{n-3} \cup S) = g(S_0 \cup \ldots \cup S_{n-3} \cup S \cup T)$, since no elements in $T$ are important in $S_0 \cup \ldots \cup S_{n-3} \cup S \cup T$. Also, by monochromaticity of $\mathcal{B}$, $g(S_0 \cup \ldots \cup S_{n-3} \cup S \cup T) = c$, so $g(S_0 \cup \ldots \cup S_{n-3} \cup S) = c$, hence proving the Claim.

\bigskip

Now, we define $f: \Nat \to \Nat$ as $f(i) = \max(B_q)$, where $q$ is minimal such that $g(B_p \cup B_{p+1} \cup B_{p+2} \cup \ldots \cup B_{p+n-3} \cup B_q) =c$, with $p$ minimal such that $i < \min(B_p)$ and $B_p < B_{p+1} < B_{p+2} < \ldots < B_{p+n-3} < B_q$. Notice that $q$ exists by Claim \ref{clm:shorter-union}. Also, $f$ has the Property $P$, i.e., each decreasible $i$ is decreased by some $j \leq f(i)$.

In order to prove this, assume by way of contradiction that $i$ is decreasible and $p,q$ are minimal such that $i < \min(B_p)$, $B_p < B_{p+1} < \ldots < B_{p+n-3} < B_q$ and $g(B_p \cup B_{p+1} \cup \ldots \cup B_{p+n-3} \cup B_q) = c$, but $i$ is decreasible only by numbers larger than $f(i) = \max(B_q)$. 

By strong $\Sigma^0_1$-bounding, let $\ell$ be the actual upper bound of the minimal indexes decreasing all the decreasible $j < \min(B_p)$. Since $\mathcal{B}$ is an infinite block sequence, there exists $B \in \mathcal{B}$ such that $\min(B) > \ell > \max(B_q)$. Now, consider:
$$
B_p \cup \ldots \cup B_{p+n-3} \cup B_q \cup B = \{b_0, \ldots, b_r, b_{r+1}, \ldots b_t\},$$
where $b_0 = \min(B_p) = \min(B_p \cup \ldots \cup B_{p+n-3} \cup B_q \cup B)$, $b_r = \max(B_q)$ and $b_{r+1} = \min(B)$. Clearly, $j \leq t$ is important in $B_p \cup \ldots \cup B_{p+n-3} \cup B_q \cup B$ if and only if either $j \leq r$ and $j$ is important in $B_p \cup \ldots \cup B_{p+n-3} \cup B_q$, or $j = r + 1$; hence, $g(B_p \cup \ldots \cup B_{p+n-3} \cup B_q) \neq g(B_p \cup \ldots \cup B_{p+n-3} \cup B_q \cup B) = c$, contra our assumption that $g(B_p \cup \ldots \cup B_{p+n-3} \cup B_q) = c$.
\end{proof}

We obtain the following immediate corollary. 

\begin{corollary}
Let $n \geq 3, k \geq 2$. Over $\RCA_0$, $\FUT^{=n}_k$ implies $\ACA_0$. 
\end{corollary}

\begin{proof}
From Theorem \ref{prop:RTomega} above and Theorem 3.1 in \cite{CKLZ:20}.
\end{proof}

The proof of Theorem \ref{prop:RTomega} can be easily adapted to $\FUT^{\leq 2}_k$ (in place of $\FUT^{=n}_k$). Yet, while $\FUT^{=n}_k$ is provably equivalent to $\ACA_0$ and so Theorem \ref{prop:RTomega} is an optimal result, we do not know the actual strength of $\FUT^{\leq 2}_k$, since we only know that $\ACA_0 \leq \FUT^{\leq 2}_k \leq \FUT \leq \ACA_0^+$. Hence, extending the above approach -- adapted to $\FUT^{\leq 2}_k$ -- to stronger well-ordering principles would improve the known lower bound on $\FUT^{\leq 2}_k$ and, a fortiori, on full Hindman's Theorem.

\section{Conclusion and perspectives}
We have presented a new approach for proving implications and Weihrauch reductions
from Ramsey-theoretic theorems to well-ordering principles (at the level of 
$\ACA_0$, $\ACA_0'$ and $\ACA_0^+$). 

This approach is inspired by Loebl and Ne\v{s}etril's independence proof of Paris-Harrington principle from Peano Arithmetic \cite{Loe-Nes:92}, was first proposed by the first and third author in \cite{Car-Zda:RT}. It yields elegant combinatorial proofs of the implications $\RT^3_2\to \ACA_0$, $\forall n \RT^n \to \ACA_0'$ and $\RT^{!\omega}_2 \to \ACA_0^+$ over
$\RCA_0$. The method also ensures uniform computable reductions. 

$\RT^{!\omega}_2$ generalizes to a Ramsey's Theorem for bicolorings of exactly $\alpha$-large sets \cite{Far-Neg:08}, and we conjecture that the method presented here can be extended to relate such general version of the theorem to the systems $\Pi_{\omega^\beta}^0$-$\CA_0$ for every $\beta \in \omega^{\textsc{ck}}$ by using
the characterization of the latter systems in terms of the well-ordering preservation 
principles $\forall \LoX (\WOP(\LoX \to \pmb{\varphi}(\beta,\LoX)))$ \cite{Mar-Mon:11}. 

Furthermore, we established a Weihrauch reduction of the well-ordering principle
$\WOP(\LoX\to\boldomega^\LoX)$ to Hindman's Theorem for $k$-colourings and sums/unions of exactly $n$ numbers/sets, for any $n \geq 3$ and $k \geq 2$. Such a direct
connection between Hindman's Theorem and well-ordering principles might be fruitful for assessing
the strength of Hindman-type principles.

\end{document}